\documentclass[11pt]{amsart}
\usepackage{amsmath,amssymb,amsthm}
\usepackage{graphics}
\usepackage{graphicx}
\usepackage{epstopdf}

\newtheorem{theorem}{Theorem}
\newtheorem{lemma}{Lemma}

\newtheorem{corollary}{Corollary}
\newtheorem{claim}{Claim}

\newcommand{\Z}{\mathbb{Z}}

\newcommand{\Au}{\mathrm{Aut}}

\begin{document}

\parindent = 0cm
\parskip = .3cm

\title[Quotients by Wreath Products]{Symmetric Chain Decompositions of\\Quotients of Chain Products by Wreath Products}

\author{Dwight Duffus}
\address {Mathematics \& Computer Science Department\\
   Emory University, Atlanta, GA  30322, USA}
   \email[Dwight Duffus]{dwight@mathcs.emory.edu}
   
\author{Kyle Thayer}
\address {Mathematics \& Computer Science Department\\
   Emory University, Atlanta, GA  30322, USA}
\email[Kyle Thayer]{kyle.thayer@gmail.com}



\begin{abstract}
Subgroups of the symmetric group $S_n$ act on powers of chains $C^n$ by permuting coordinates, and
induce automorphisms of the ordered sets $C^n$.  The quotients defined are candidates for symmetric
chain decompositions.  We establish this for  some families of groups in order to enlarge the collection of
subgroups $G$ of the symmetric group $S_n$ for which the quotient $B_n/G$ obtained from the $G$-orbits
on the Boolean lattice $B_n$ is a symmetric chain order.  The methods are also used to provide an elementary
proof that quotients of powers of SCOs by cyclic groups are SCOs.
\end{abstract}

\maketitle

\section{Introduction}\label{s:intro}

We are interested in one of the most well-studied symmetry properties of a finite partially ordered set -- possessing a partition into symmetric chains.  We want to determine circumstances under which having of a symmetric chain decomposition is preserved by quotients.

Here we restrict our attention to finite ordered sets $P$ with minimum element, denoted $0_P$, in which all maximal chains have the same length.
Such $P$ have
a rank function $r= r_P$ defined as follows:  for all $x \in P$, $r(x)$ is the maximum length $l(C) = |C| - 1$ over all chains 
$C \subseteq P$ with  maximum element $x$.  The {\it rank} $r(P)$ of $P$ is 
the maximum of $r(x)$ over all $x \in P$ and we call such $P$ {\it ranked} partial orders.

In a ranked order $P$ a chain $x_1 < x_2 < \dots < x_k$ is {\it saturated} if for each $i$ there is no $z$ such that
$x_i < z < x_{i + 1}$.  Call the saturated chain  \emph{symmetric} if 
$r(x_1) + r(x_k) = r(P)$.   A \emph{symmetric chain decomposition}, an \emph{SCD}, of $P$ is a partition of $P$ into 
symmetric chains.  If $P$ has an SCD, call $P$ a {\it symmetric chain order}, an {\it SCO}.  
Here, we are concerned with a particular family of SCOs, products of finite chains, and whether quotients of
these are also SCOs.

For any partially ordered set $P$ and any subgroup $G$ of $\Au (P)$ the automorphism group of $P$, let $P/G$ 
denote the quotient poset.  That is, the elements of $P/G$ are the orbits induced by $G$, denoted by $[x]_G$, or just $[x]$
when no confusion should arise.  And, $[x] \le [y]$ in $P/G$ 
if there are $x' \in [x]$ and $y' \in [y]$ such that $x' \le y'$ in $P$.    Canfield and Mason \cite{CM} conjectured that for 
$P = B_n$, the Boolean lattice of all subsets  of an $n$-element set ordered by containment, $B_n / G$ is an SCO 
for all subgroups $G$ of $\Au (B_n)$.  (Observe that $\Au (B_n) \cong S_n$, the symmetric group on $[n] := \{1, 2, \ldots, n\}$.)   A more general problem is to investigate
conditions on an SCO $P$ and a subgroup $G$ of $\Au (P)$ under which $P/G$ an SCO.

In studying Venn diagrams, 
Griggs, Killian and Savage \cite{GKS} explicitly constructed an SCD of the quotient $B_n/G$ for $n$ prime and given 
that $G$ is generated by a single $n$-cycle.  They asked if this {\it necklace} poset is an SCO for arbitrary $n$.  
Jordan \cite{KKJ} proved that it is by constructing an SCD of the quotient on the SCD in $B_n$ based on an SCD obtained by
Greene and Kleitman \cite{GK} for the Boolean lattice.   P. Hersh and A. Schilling \cite{HS} gave another proof of Jordan's result with an explicit construction of an SCD in $B_n/\mathbb{Z}_n$ via invention of a cyclic version of Greene and Kleitman's bracketing procedure.
We showed that $B_n/G$ is an SCO provided that $G$ is generated by 
powers of disjoint cycles \cite{DMT}.  While this generalizes Jordan's result somewhat,  the method of proof is likely
more interesting in that we construct an SCD of the quotient by refining, or pruning, the Greene-Kleitman SCD in a more direct way than Jordan.

Results can be extended in several directions.  Dhand \cite{D} has shown that if $P$ is an SCO and $\mathbb{Z}_n$ acts
on $P^n$ in the usual way (by permuting coordinates) then $P^n/\mathbb{Z}_n$ is an SCO.  (This is the same as
considering $P^n/G$ where $G$ is generated by an $n$-cycle.) His methods are algebraic.  As a special case,
in some sense the base case of his argument, for any chain $C$, $C^n/G$ is an SCO for $G$ generated by an
$n$-cycle.  We have a somewhat more general version of this special case, that $C^n/G$ is an SCO for $G$ generated
by powers of disjoint cycles \cite{DMT}.  The proof again proceeds by pruning the Greene-Kleitman SCD to give an SCD of the quotient.

Quotients of the form $(C_1 \times C_2 \times \cdots \times C_n) /K$, where each $C_i$ is a chain, are interesting for 
several reasons.  First, it follows from a result of Stanley \cite{RPS3} that  for any chain product and all subgroups $K$
of its automorphism group, $(C_1 \times C_2 \times \cdots \times C_n) /K$
is rank-symmetric, rank-unimodal, and strongly Sperner (see, for instance, \cite{E} or \cite{KKJ} for definitions).  An SCO has
these three properties.  While these conditions are not sufficient to insure an SCD, Griggs \cite{JRG} showed that a 
ranked ordered set with the LYM property, rank-symmetry and rank-unimodality is an SCO.   And it is the case that
products of chains have the LYM property.  (We shall see, at the end of Section \ref{s:examples}, that Stanley \cite{RPS}
was first to ask if certain quotients of $B_n$ are SCOs.)

Second, the automorphism group of a product of chains consists of just those maps that permute coordinates corresponding
to chains of the same length.  To be more precise, for 

\centerline{$P = \prod_{i = 1}^{m} C_i^{n_i},  \ C_j \ncong C_k \ \text{for} \ j \ne k,$}
and $\phi \in \Au(P)$ there exist $\phi_i \in \Au(C_i^{n_i})$ and $\sigma_i \in S_{n_i}$, $i = 1, 2, \ldots, m$, such that 
$\phi = ( \phi_1, \phi_2, \ldots, \phi_m )$
and for $\overline{x} = (x_1 , x_2, \ldots, x_{n_i} )\in  C_i^{n_i}$,
$$\phi_i (\overline{x}) = (x_{\sigma_i^{-1}(1)}, x_{\sigma_i^{-1}(2)}, \ldots, x_{\sigma_i^{-1}(n_i)} ).$$
In particular, $\Au (P)$ is identified with a subgroup of $S_{n_1} \times S_{n_2} \times \cdots \times S_{n_m}$ that is itself a product
of wreath products (see Section \ref{s:mainresult}).   These observations are
justified by examining the action of an automorphism of a chain product on the atoms of the product.  They are also a 
consequence of results of Chang, J\'onsson and Tarski \cite{CJT} on the refinement property for product decompositions of
partially ordered sets. 

Third, since a product of SCOs is an SCO (first proved by Katona \cite{K} and rediscovered more than once \cite{E}), whenever a subgroup $K$ of $\Au(P)$ factors into subgroups $K_i$ of $\Au(C_i^{n_i})$,
$P/K$ is an SCO provided that each $C_i^{n_i}/K_i$ is an SCO.  Thus, $P/K$ is an SCO whenever each of the factors
$K_i$ is generated by powers of disjoint cycles, as noted in \cite{DMT} in Corollary 7.

Finally, one can use the fact that $(C_1 \times C_2 \times \cdots \times C_n )/K$ has an SCD for certain groups $K$
to prove that $B_n /G$ has an SCD for groups $G$ that are wreath products involving these $K$.  This is described
precisely in Section \ref{s:mainresult}.    The main results are stated and proved there.   In Section \ref{s:dhand} we show
how the methods developed to prove Theorem \ref{T:main} can be used to prove, in a very elementary way,  the result of 
Dhand \cite{D} mentioned above.  Examples and questions are
presented in Section \ref{s:examples}.   

\section{The Main Result}\label{s:mainresult}

We requires some terminology and notation to state the main result.  We use standard ordered set notation.  See \cite {DM} for 
background on permutation groups.  

For this section, let $n = k t$ for $n, k, t \in \mathbb{N}$ and let $B_n$ denote the Boolean lattice of subsets of
the $n$-element set $[k] \times [t]$ of ordered pairs.  Let $S_X$ denote symmetric
group on any set $X$.  It is convenient to use $S_n$ in place of $S_{ [k] \times [t]}$, while using $S_k$ and $S_t$
for the symmetric groups on $[k]$ and $[t]$, respectively.  Let 
\begin{equation}\label{e:partition}
[k]\times[t] = N_1 \cup N_2 \cup \cdots \cup N_t \ \  \text{where} \ \ N_r = [k] \times \{r\}, \ (r = 1, 2, \ldots, t).
\end{equation}

Given subgroups $K$ of $S_k$ and $T$ of $S_t$, the wreath product  $G =  K \wr T$ of $K$ by $T$ is the set of all
$\phi \in S_n$ defined as follows: given $\overline{\rho} = (\rho_1, \rho_2, \ldots, \rho_t ) \in K^t$ and $\tau \in T$, for all
$(i, r) \in [k]\times[t]$,
\begin{equation}\label{e:phi}
\phi (i, r) = (\rho_r (i), \tau(r) ).
\end{equation}  

Wreath products allow us to extend the methods in \cite{DMT} to larger families of groups.
\begin{theorem}\label{T:main}
Let $n, k, t$ be positive integers and  $n = kt$.   Then $B_n/G$ is a symmetric chain order for any subgroup $G$ of 
$\Au (B_n)$ defined as follows: $G = K \wr T$ where $K$ is a subgroup of $S_k$, $T$ is a subgroup of $S_t$, and
both $K$ and $T$ are generated by powers of disjoint cycles.
\end{theorem}

The three subsections below are organized as follows.   In {\bf\S\ref{ss:required}}, we specify the properties of $K$ and $T$ that are needed in our proof that for $G = K \wr T$, $B_n/G$ is an SCO.  In {\bf \S\ref{ss:sufficiency}}, we show that these properties do indeed suffice.  In
{\bf \S\ref{ss:powers}}, we prove that if subgroups $K$ and $T$ are each generated by powers of disjoint cycles then they do have the
required properties.

\subsection{The required properties}\label{ss:required}
Let us assume that $G$ is a subgroup of $S_n$ and $G = K \wr T$ with $K$ a subgroup of $S_k$ and $T$ a subgroup of
$S_t$.    To derive the requisite properties, some notation and a few observations are needed.

Given a partially ordered set $P$ and subsets $P_i$, $i = 1, 2, \ldots, m$, use $P = \sum_{i = 1}^m P_i$ to mean
that $P$ is partitioned by the family of $P_i$'s and the order on each $P_i$ is the restriction of the order on $P$, 
with no information given about the order relations between elements of distinct  $P_i$'s.  For a subposet $Q$ of
a ranked partially ordered set $P$, $Q$ is said to be {\it saturated} if for all $x, y \in Q$, $x$ is covered by $y$ in
$Q$ implies that $x$ is covered by $y$ in $P$.   Given a saturated $Q \subseteq P$ with minimum $0_Q$ and
maximum $1_Q$, say that  $Q$ is {\it symmetric} in $P$ if $r_P( 0_Q ) + r_P(1_Q) = r_P(P).$
\noindent\footnote{As is customary, we define ``symmetric" only for saturated subposets.  However, we usually use the phrase
``saturated and symmetric", or some variant, to emphasize the two aspects of the property.}

\begin{lemma}\label{l:partition}
Let $P$ be a partially ordered sets and let $P_1, P_2, \ldots, P_m$ be saturated, symmetric subsets of $P$.  If
$P = \sum_{i = 1}^m P_i$ and each $P_i$ is an SCO then $P$ is an SCO.
\end{lemma}
\begin{proof}
Each chain $C$ in an SCD of a $P_i$ is a symmetric, saturated chain in $P_i$ and each $P_i$ is symmetric and saturated in $P$.
Hence, $C$ is symmetric and saturated in $P$.  Thus, the set of all these chains, over all $i$, provides an SCD of $P$.
\end{proof}

For all $X \subseteq [n]$, let $X_r = X \cap N_r$, $r = 1, 2, \ldots, t$ (see (\ref{e:partition})).
With the notation in (\ref{e:phi}), let $K'$ be the subgroup of $G$ consisting of all $\phi$ for which $\tau$ is the
identity on $[t]$.  This is 
the base group of
the wreath product and is isomorphic to $K^t$.   Let $K_r$ denote the copy of $K$ acting on $N_r$ and for any set
$N$, let $B(N)$ denote the Boolean lattice of all subsets of $N$.  Then\\[-.6cm]
\begin{equation}\label{e:inside}
B_n/K' \ \cong \ \prod_{r = 1}^t B(N_r)/K_r \ \cong \ (B_k/K)^t .
\end{equation}
The first isomorphism, say $F$, in (\ref{e:inside}) is the map $[X]_{K'} \xrightarrow{F}  ([X_1]_{K_1}, [X_2]_{K_2}, \ldots, [X_t]_{K_t})$, where
$[X]_{K'}$ and $[X_r]_{K_r}$ denote the orbit of $X$ under $K'$ and the orbit of $X_r$ under $K_r$, respectively.

Here is the property of the subgroup $K$ of $S_k$ needed in our proof:\\

\centerline{[{\bf K}]:   $B_k / K$ has an SCD.}

With an SCD guaranteed by [{\bf K}], we now derive the required property of $T$.  

Suppose that $C_1, C_2, \ldots, C_s$ constitute an SCD of $B_k/K$ and that $C_1^r, C_2^r, \ldots, C_s^r$ is the 
corresponding SCD of $B(N_r)/K_r$, $r = 1, 2, \ldots, t$.   Then these SCDs define a partition of the product $\prod_{r=1}^t B(N_r)/K_r$ into {\it grids}, that is, products of chains:
\begin{equation}\label{e:grids}
\prod_{r = 1}^t B(N_r)/K_r \ = \ \prod_{r = 1}^t \left( \sum_{j = 1}^s C_j^r \right) \ = \!
\sum_{\overline{j} = (j_1, j_2, \ldots, j_t ) \in [s]^t} C_{j_1}^1 \times C_{j_2}^2 \ldots \times C_{j_t}^t .
\end{equation}
Note that each grid $C(\overline{j}) =  C_{j_1}^1 \times C_{j_2}^2 \ldots \times C_{j_t}^t$ is itself an SCO 
(\cite {BTK}) and
is a symmetric, saturated subset of  $\prod_{r = 1}^t B(N_r)/K_r$.   Following Lemma \ref{l:partition}, the family of all chains from the SCDs of the $C(\overline{j})$s
and the first isomorphism in (\ref{e:inside}) yield an SCD for $B_n/K' $.

The group $T$ acts on $[s]^t$ as follows: for $\tau \in T$ and $\overline{j} = (j_1, j_2, \ldots, j_t ) \in [s]^t$,
$$\tau(\overline{j}) = (j_{\tau^{-1}(1)}, j_{\tau^{-1}(2)}, \ldots, j_{\tau^{-1}(t)} ).$$  
Let $T_{\overline{j}}$ denote the stabilizer of ${\overline{j}}$ -- the subgroup of all $\tau \in T$ such that 
$\tau(\overline{j}) = {\overline{j}}$.  Of course, each $\tau \in T$ also induces a permutation of the family of all $C(\overline{j})$, 
$\overline{j} \in [s]^t$, via $C(\overline{j}) \to C(\tau(\overline{j}))$.    In fact, each $\tau \in T$ induces ``local isomorphisms"
among the grids $C(\overline{j})$, $\overline{j} \in [s]^t$, as follows.

Given $X \subseteq [k] \times [t]$ and $X_r = X \cap N_r$, let $X_{r,p} = \{ (i, p) | (i, r) \in X_r \}$.  Thus, $X_r$ and $X_{r,p}$ 
have the same first projection into $[k]$.  Also, for brevity, let $[X_r]$ replace $[X_r]_{K_r}$, let $[X_{r,p}]$ replace $[X_{r,p}]_{K_p}$,
and let $\overline{X} =  ([X_1], [X_2], \ldots , [X_t])$.

Define $\widehat{\tau}: C(\overline{j}) \to C(\tau(\overline{j}))$ by
\begin{equation}\label{e:tau}
\widehat{\tau} ([X_1], [X_2], \ldots , [X_t]) = ([X_{{\tau}^{-1}(1), 1}], [X_{{\tau}^{-1}(2), 2}], \ldots, [X_{{\tau}^{-1}(t), t}]) .
\end{equation}
Then $\widehat{\tau}$ is well-defined and is an order isomorphism because:
\begin{align*}
\widehat{\tau}(\overline{X}) = \widehat{\tau} ([X_1], [X_2], \ldots , [X_t]) \le \widehat{\tau} ([Y_1], [Y_2], \ldots , [Y_t]) \ \text{in} \ C(\tau(\overline{j})) &\iff \\
[X_{{\tau}^{-1}(r), r}] \le [Y_{{\tau}^{-1}(r), r}] \ \text{in} \  C_{j_{\tau^{-1} (r)}}^r  \ (r = 1, 2, \ldots, t) &\iff \\
[X_{{\tau}^{-1}(r)}] \le [Y_{{\tau}^{-1}(r)}]  \ \text{in} \ C_{j_{\tau^{-1} (r)}}^{\tau^{-1} (r)}  \ (r = 1, 2, \ldots, t) &\iff \\
 ([X_1], [X_2], \ldots , [X_t]) \le ([Y_1], [Y_2], \ldots , [Y_t]) \ \text{in} \  C(\overline{j}).
\end{align*}

Whenever $\tau \in T_{\overline{j}}$,  $\widehat{\tau}: C(\overline{j}) \to C(\overline{j})$ is an automorphism of
 $C(\overline{j})$.
 
Finally, we can state the property of $T$ we need.

[{\bf T}]:  With the SCD $C_1, C_2, \ldots, C_s$ of $B(k)/K$, for each $\overline{j} \in [s]^t$, 
$C(\overline{j})/T_{\overline{j}}$ is an SCO.

\subsection{Proof of Theorem \ref{T:main} from [K] and [T]}\label{ss:sufficiency}
In this subsection, we assume that the subgroups $K$ and $T$ of $S_k$ and $S_t$, respectively, satisfy [{\bf K}] and [{\bf T}].
We have the partition of $\prod_{r = 1}^t B(N_r)/K_r$ into symmetric, saturated grids  $C(\overline{j}), \ \overline{j} \in  [s]^t$, given
in (\ref{e:grids}) and know that the union of the SCDs of these grids gives an SCD of  $\prod_{r = 1}^t B(N_r)/K_r$.

Select a representative from each orbit in $[s]^t$ under $T$, say the lexicographically least, and let $J$ be the set of these
representatives.  The family $\{ C(\overline{j}) \ | \ \overline{j} \in J \}$ seems to be a promising source of an SCD of $B_n/G$.
Indeed, for $X = \cup_{r = 1}^t X_r \subseteq [n]$, each $X_r$ determines an index $j_r$ via $[X_r] \in C_{j_r}^r$ in the SCD of
$B(N_r)/K_r$.   There is a unique $\overline{j} \in [s]^t$ such that $\overline {X} = ([X_1], [X_2], \ldots , [X_t]) \in C(\overline{j})$.   Thus, there is
exactly one $\overline{j'} \in J$ such that $\widehat{\tau} (\overline{X}) \in C(\overline{j'})$ for some $\tau \in T$.

Creating an SCD for $B_n/G$ begins with verifying this claim.

\begin{claim}\label{cl:bijection}
Let $\Phi: B_n/G \to \sum_{\overline{j} \in J} C(\overline{j})/T_{\overline{j}}$ be defined by
$$\Phi([Y]_G) = [\overline{X}]_{T_{\overline{j}}}$$
where $X \in [Y]_G$ and $\overline{X} \in C(\overline{j})$ for some $\overline{j} \in J$.  Then $\Phi$ is a bijection.
\end{claim}

\begin{proof}
Let's begin by verifying the following equivalence:   for all $U, V \subseteq [k] \times [t]$,
\begin{equation}\label{e:mainpoint}
\phi(U) = V \ \text{for some} \ \phi \in G \ \text{if and only if} \ \widehat{\tau}(\overline{U}) = \overline{V}  \ \text{for some} \ \tau \in T .
\end{equation}

Suppose that $\phi \in G$ corresponds to $\overline{\rho} \in K^t$ and $\tau \in T$ as in (\ref{e:phi}).  We need a description of
$\phi(W)$ for any $W \subseteq [k] \times [t]$: with $W_r' = \{ i \in [k] | (i,r) \in W_r \}$, 
$$\phi(W) \ = \ \bigcup_{r = 1}^t \left( \rho_{\tau^{-1}(r )}( W_{\tau^{-1}(r)}') \times \{ r \} \right) .$$
With this notation we have
\begin{align*}
\phi(U) = V \ &\iff \bigcup_{r = 1}^t \left( \rho_{\tau^{-1}(r )}( U_{\tau^{-1}(r)}') \times \{ r \} \right) = \bigcup_{r = 1}^t V_r \\
&\iff  \ \rho_{\tau^{-1} (r)}(U_{{\tau}^{-1}(r)}')   \times \{ r \} = V_r , \ r = 1, 2, \ldots, t \\
&\iff  \ [U_{\tau^{-1}(r), r}] = [V_r],   \ r  = 1, 2, \ldots, t\\
&\iff \  \widehat{\tau}(\overline{U}) = \overline{V} .
\end{align*}
The first equivalence is from the description of the wreath product in (\ref{e:phi}).  The second is obvious.  The
third is the definition of the $K_r$-orbit.  The last is the definition of $\widehat{\tau}$ in (\ref{e:tau}).

Let us see that $\Phi$ is well-defined and injective.  Suppose that $X$ and $Y$ are as in the statement of the claim
and that $V \in [U]_G$, $V \in C_{\overline{j'}}$ for some $\overline{j'} \in J$ and $\Phi([U]_G) = [\overline{V}]_{T_{\overline{j'}}}$.
Suppose that $[U]_G = [Y]_G$.
Then $\phi(X) = V$ for some $\phi \in G$ so, by (\ref{e:mainpoint}), $\widehat{\tau}(\overline{X}) = 
\overline{V}$ for some $\tau \in T$.  Then $\tau(\overline{j}) = \overline{j'}$.  Since $\overline{j},\overline{j'} \in J$, a system of representatives
of the $T$-orbits in $[s]^t$, $\overline{j} = \overline{j'}$ and $\tau \in T_{\overline{j}}$.  Hence, $ [\overline{X}]_{T_{\overline{j}}} =
 [\overline{V}]_{T_{\overline{j}}}$, that is, $\Phi([U]_G) = \Phi([Y]_G).$  Thus, $\Phi$ is well-defined.  The converse is easily argued,
 also based on (\ref{e:mainpoint}), establishing that $\Phi$ is injective.
 
 It is obvious that $\Phi$ is surjective: any element of  $\sum_{\overline{j} \in J} C(\overline{j})/T_{\overline{j}}$ is the form
 $[\overline{X}]_{T_{\overline{j}}}$, for some $\overline{X} \in C({\overline{j}})$, ${\overline{j}} \in J$.  Then $X \subseteq [k] \times [t]$
 and $\Phi([X]_G) = [\overline{X}]_{T_{\overline{j}}}$.  \end{proof} 
 
 We began this subsection assuming that 
 
 [{\bf T}]:  With the SCD $C_1, C_2, \ldots, C_s$ of $B(k)/K$, for each $\overline{j} \in [s]^t$, 
$C(\overline{j})/T_{\overline{j}}$ is an SCO.

Let  $\overline{j} \in J$ and let $\mathcal{C}$ be a symmetric chain in an SCO of $C(\overline{j})/T_{\overline{j}}$.

\begin{claim}\label{cl:quotient}
The inverse image $\Phi^{-1} (\mathcal{C})$ of $\mathcal{C}$ under $\Phi$ is a symmetric, saturated chain in $B_n/G$.
\end{claim}

\begin{proof}
Observe that for all ranked orders $P$, and all subgroups $G$ of $\Au(P)$, and all $x \in P$, $r_{P/G} ([x]_G) = r_P(x)$. Thus,
for all $[\overline{X}]_{T_{\overline{j}}} \in C(\overline{j})/T_{\overline{j}}$,
$$r_{C(\overline{j})/T_{\overline{j}}} ([\overline{X}]_{T_{\overline{j}}}) = r_{C(\overline{j})} (\overline{X}) \ \text {and} \  r_{B_n/K'} ([X]_{K'}) = r_{B_n/G} ([X]_{G}) .$$
Recall $C(\overline{j})$ is a symmetric, saturated subset of $\prod_{r = 1}^t B(N_r)/K_r \cong B(n)/K'$ and that $r_{B_n/K'} = r(B_n) = n$.  If $r_{C(\overline{j})} = n_{\overline{j}}$ then we have 
$$ r_{B_n/K'} ([X]_{K'}) = 1/2(n - n_{\overline{j}}) + r_{C(\overline{j})} (\overline{X}).$$
Since $\mathcal{C}$ is a symmetric saturated chain in an SCD of $C_{\overline{j}}/T_{\overline{j}}$, the ranks of its elements form an interval
$k, k+1, \dots, n_{\overline{j}} - k$, for some $k$.  Thus, the ranks of the elements of  $\Phi^{-1} (\mathcal{C})$ in $B_n/G$ form the interval
$$1/2(n - n_{\overline{j}}) + k, 1/2(n - n_{\overline{j}}) + k + 1, \ldots, 1/2(n - n_{\overline{j}}) +  n_{\overline{j}} - k = n - (1/2(n - n_{\overline{j}}) + k),$$
a symmetric interval of ranks in $B_n/G$.

Suppose that $[\overline{X}]_{T_{\overline{j}}} \le [\overline{Y}]_{T_{\overline{j}}}$ in $\mathcal{C}$.  Then there is some $\overline{Y'} = 
\widehat{\tau}(\overline{Y})$, for some $\tau \in T_{\overline{j}}$ such that $\overline{X} \le \overline{Y'}$ in $\prod_{r = 1}^t B(N_r)/K_r$.
Then $[X]_{K'} \le [Y']_{K'}$ and, so,  $[X]_{G} \le [Y']_{G} = [Y]_G$.  Hence,  $\Phi^{-1} (\mathcal{C})$ is a symmetric saturated chain in
$B_n/G$.  \end{proof}

The collection of the symmetric, saturated chains $\Phi^{-1} (\mathcal{C})$, over all $\mathcal{C}$ in the given SCDs of  all $C(\overline{j})/T_{\overline{j}}$, $\overline{j} \in J$, provides an SCD of $B_n/G$.

\subsection{Powers of disjoint cycles suffice}\label{ss:powers}

Let's see that if both $K$ and $T$ are generated by powers of disjoint cycles, in $S_k$ and $S_t$, respectively, then both {\bf [K]} and
{\bf [T]} hold.   We use this result from \cite{DMT} (Corollary 7) and \cite{D}.

\begin{lemma}\label{l:cycles}
Let $P$ be a product of chains and let $K$ be a subgroup of $\Au(P)$ that is generated by powers of disjoint cycles.  Then $P/K$ is an SCO.
\end{lemma}

Of course {\bf [K]} holds, when we assume $K$ is generated by powers of disjoint cycles.

The following observation show  that {\bf [T]} holds, assuming $T$, too, is generated by powers of disjoint cycles.

\begin{lemma}\label{L:stabilizer}  Let $s, t$ be positive integers and let $T$ be a subgroup of $S_t$ that is generated by powers of disjoint cycles.   Then for all $\overline{j} \in [s]^t$, the stabilizer $T_{\overline{j}}$ of $\overline{j}$ in the action of $T$ on $[s]^t$ is also generated by powers of disjoint cycles.
\end{lemma}

\begin{proof}
Let $\overline{j} \in [s]^t$ and $\tau \in T$.  Then $\tau \in T_{\overline{j}}$ if and only
if  for all $u, v \in [t]$, $\tau (u) = v$ implies that $j_u = j_v$.  In other words, the partition of $[t]$ defined by the orbits of $[t]$
under $\tau$ must refine the partition of $[t]$ defined $\overline{j}$.   For each $\tau \in T$ there is a minimum $m$ such that
$\tau^m \in  T_{\overline{j}}$.  This is because $\tau^k \in  T_{\overline{j}}$ for at least one $k$, the order of $\tau$,  and if
$\tau^a, \tau^b \in  T_{\overline{j}}$ then with $d = \gcd (a, b)$, $\tau^d \in  T_{\overline{j}}$.

Suppose that $T = \langle \sigma_1^{r_1},  \sigma_2^{r_2}, \ldots,  \sigma_m^{r_m} \rangle$ and $d_i$ is minimum such that
$(\sigma_i^{r_i})^ {d_i} \in  T_{\overline{j}}$ for $i = 1, 2, \ldots, m$.  Then $T_{\overline{j}} = \langle \sigma_1^{r_1 d_1},  
\sigma_2^{r_2 d_2}, \ldots,  \sigma_m^{r_m d_m} \rangle$.  \end{proof}

This completes the proof of Theorem \ref{T:main}.

There are certainly circumstances, other than $K$ and $T$ each being generated by powers of disjoint cycles, under which the
properties {\bf [K]} and {\bf [T]} hold.  Here is an instance.  If $K$ is either the full symmetric group $S_k$ or the alternating subgroup
$A_k$ the the quotient $B(k)/K$ is a $k+1$-element chain -- evidently an SCO!  Let's denote this by $C$.  In this case property {\bf [T]} amounts to 
having that $C^t/T$ is an SCO.   Therefore we have the following, which we state as a corollary since it is a direct consequence of the
proof of Theorem \ref{T:main} and Corollary 7 in \cite{DMT}.

\begin{corollary}
Let $n, k, t$ be positive integers and  $n = kt$.   Then $B_n/G$ is a symmetric chain order for any subgroup $G$ of 
$\Au (B_n)$ defined as follows: $G = K \wr T$ where $K$ is either $S_k$ or the alternating subgroup $A_k$, and $T$ is a subgroup of $S_t$
generated by powers of disjoint cycles.
\end{corollary}

\section{More General SCOs}\label{s:dhand}

The methods used in Section \ref{s:mainresult} provide what appears to be a new and elementary proof
of Dhand's result on quotients of SCOs by a cyclic group \cite{D}.   We also use Lemma \ref{l:cycles} (\cite{DMT}, Corollary 7).

\begin{theorem}\label{T:dhand}
Let $P$ be a symmetric chain order and let $G$ be a subgroup of $S_n$ generated by powers of disjoint cycles.  Then
$P/G$ is a symmetric chain order.
\end{theorem}

\begin{proof}
Let $C_1, C_2, \ldots, C_s$ form an SCD of $P$.   As noted after (\ref{e:grids}) in {\bf \S2.1}, the family
of grids 
$$C(\overline{j}) = C_{j_1} \times C_{j_2} \times \cdots \times C_{j_n}, \ \overline{j} = (j_1, j_2, \ldots , j_n) \in [s]^n$$
provides a partition of $P^n$ into symmetric, saturated subsets.   The group $G$ acts on $[s]^n$ as before, $\sigma(j_1, j_2, \ldots, j_n) =
(j_{\sigma^{-1}(1)},  j_{\sigma^{-1}(2)}, \ldots, j_{\sigma^{-1}(n)})$, and defines these mappings:

\qquad $\sigma$ permutes the members of the family $\{ C(\overline{j}) \ | \ \overline{j} \in [s]^n \}$;

\qquad $\sigma: P^n \to P^n$, $\sigma(\overline{p}) = (p_{\sigma^{-1}(1)}, p_{\sigma^{-1}(2)}, \ldots, p_{\sigma^{-1}(n)}, )$, $\overline{p} = (p_1, p_2, \ldots, p_n) \in P^n$, 

\qquad\qquad is an automorphism; and,

\qquad for each $\overline{j} \in [s]^n$, the restriction of $\sigma$ is a local isomorphism of $C(\overline{j})$ to $C(\sigma(\overline{j}))$.

Let $J \subseteq [s]^n$ be a system of representatives of the $G$-orbits of $[s]^n$.   Note that for each $\overline{j} \in [s]^n$ and $\sigma \in G_{\overline{j}}$, the stabilizer of $\overline{j}$ in $G$, the restriction of $\sigma$ is an automorphism of $C(\overline{j})$.  

\begin{claim}\label{cl:bijection2}
Let $\Psi: P^n/G \to \sum_{\overline{j} \in J} C(\overline{j})/G_{\overline{j}}$ be defined by
$$\Psi([\overline{y}]_G) = [\overline{x}]_{G_{\overline{j}}}$$
where $\overline{x} \in [\overline{y}]_G$ and $\overline{x} \in C(\overline{j})$ for some $\overline{j} \in J$.  Then $\Psi$ is a bijection.
\end{claim}

The proof of this claim is analogous to that of  Claim \ref{cl:bijection}, with simpler notation.  If $[\overline{y'}]_G = [\overline{y}]_G$,
$\overline{x'} \in [\overline{y'}]_G$ with $\overline{x'} \in C(\overline{j'})$ for some $\overline{j'} \in J$, and  $\overline{x}, \overline{y}$ are
as in the statement of Claim \ref{cl:bijection2}, then there exist $\sigma_1, \sigma_2, \sigma_3 \in G$ such that
$$\sigma_1(\overline{x}) = \overline{y}, \sigma_2(\overline{y'}) = \overline{y}, \ \text{and} \ \sigma_3(\overline{x'}) = \overline{y'}.$$
Thus, $\sigma_1^{-1} \sigma_2 \sigma_3 (\overline{x'}) = \overline{x}$.  Since $\overline{x} \in C_{\overline{j}}$, $\overline{x'} \in C_{\overline{j'}}$,
and $\overline{j}, \overline{j'} \in J$, we have $\overline{j} = \overline{j'}$ and, thus, $\sigma_1^{-1} \sigma_2 \sigma_3 \in G_{\overline{j}}$.  This shows that $\Psi$ is well-defined.
It is easily shown to be bijective.

By Lemma \ref{L:stabilizer} and the hypothesis that $G$ is generated by powers of disjoint cycles, we have that the same is true for each $G_{\overline{j}}$.  By Lemma \ref{l:cycles}, each $C(\overline{j})/G_{\overline{j}}$ has an SCD.    Let  $\overline{j} \in J$ and let $\mathcal{C}$ be a symmetric chain in an SCO of $C(\overline{j})/G_{\overline{j}}$. That $P^n/G$ has an SCD follows from

\begin{claim}\label{cl:quotient2}
The inverse image $\Psi^{-1} (\mathcal{C})$ of $\mathcal{C}$ under $\Psi$ is a symmetric, saturated chain in $P^n/G$.
\end{claim}

Follow the proof of Claim \ref{cl:quotient} in the obvious way to establish that the set of ranks of elements of $\Psi^{-1} (\mathcal{C})$ is a
symmetric interval in the ranks of $P^n/G$ and that if $\Psi([\overline{y}]_G) \le \Psi([\overline{u}]_G)$ in $\mathcal{C}$ then 
 $[\overline{y}]_G \le [\overline{u}]_G$ in $P^n/G$.  \end{proof}

\section{Examples and Questions}\label{s:examples}

The subgroups of $S_n$ to which Theorem \ref{T:main} and the results of \cite{DMT} apply appear to be quite a restricted
class.   Unfortunately, we do not have a very satisfying description of the family of groups to which our
methods apply.   We can provide a catalog of those subgroups of $S_n$ that, by these results, do generate 
quotients of $B_n$ with SCDs for some small values of $n$.  We display the collection for $n = 6$.

First, let's review the facts we can apply.   Let $G$ be a subgroup of $S_n$.

 \qquad     $\bullet$ If $[n] = X_1 \cup X_2 \cup \cdots \cup X_m$ is a partition and $G = G_1 \times G_2 \times \cdots \times G_m$ where
     $G_i = G_{|X_i}$ then $B_n/G$ is an SCO provided that each $B(X_i)/G_i$ is an SCO.
       
 \qquad     $\bullet$  If $G = \langle \sigma^r \rangle$ for some cycle $\sigma \in S_n$ then $B_n/G$ is an SCO.
         
 \qquad    $\bullet$ Theorem \ref{T:main}.    
     
For $n = 6$, we can organize the subgroups by listing them according to the partitions of 6 given by the
orbit sizes, then refining with the cycle structure.  The list given should be all the subgroups, up to conjugation, that are handled
by the three statements above and the observation that the quotient of $B_n$ by either the symmetric group $S_n$ or the alternating group
$A_n$ is an $n+1$-element chain.

 \begin{figure}
 
    \begin{tabular}{l  l  l  r} \hline \\[-.1cm]
    {\bf Partition of 6} &  \qquad {\bf Generators} &  \qquad {\bf Description} &  \qquad {\bf Order} \\[.1cm]  \hline  \\ 
    1 + 1 + 2 + 2 &  \qquad (1 2 ), (3 4) & \qquad $ \Z_2 \times \Z_2$ &  \qquad 4 \\[.1cm] 
                            &  \qquad (1 2 )(3 4)  &  \qquad$ \Z_2$                     &  \qquad 2 \\[.2cm] \hline \\[-.2cm]
   1 + 1 + 4 &  \qquad (1 2 3 4) & \qquad $ \Z_4 $ &  \qquad 4 \\[.1cm] 
                    &  \qquad (1 2 3 4), (1 2)  &  \qquad$ S_4$                     &  \qquad 24 \\[.1cm] 
                    &  \qquad (1 2 3), (2 3 4)  &  \qquad$A_4$                     &  \qquad 12 \\[.1cm] 
                    &  \qquad (1 2 3 4), (1 3)  &  \qquad$ \Z_2 \wr \Z_2 \cong D_8$     &  \qquad  8   \\[.2cm] \hline \\[-.2cm]

   1 + 2 + 3 &  \qquad (1 2), (3 4 5) & \qquad $ \Z_2 \times \Z_3 $ &  \qquad 6 \\[.1cm]                  
                    &  \qquad (1 2), (3 4 5), (3 4)  &  \qquad$ \Z_2 \times S_3$     &  \qquad  12   \\[.2cm] \hline \\[-.2cm]
      
   2 + 2 + 2 &  \qquad (1 2), (3 4), (5 6) & \qquad $ \Z_2 \times \Z_2 \times \Z_2 $ &  \qquad 8 \\[.1cm]                  
                    &  \qquad (1 2)(3 4), (5 6)     &  \qquad$ \Z_2 \times \Z_2$     &  \qquad  4 \\[.1cm] 
                    &  \qquad (1 2)(3 4)(5 6)     &  \qquad$ \Z_2$     &  \qquad  2  \\[.2cm] \hline \\[-.2cm]                                                                                                                                                  3 + 3 &  \qquad (1 2 3), ( 4 5 6) & \qquad $ \Z_3 \times \Z_3$ &  \qquad 9 \\[.1cm]                  
                    &  \qquad (1 2 3), (1 2), (4 5 6)     &  \qquad$ S_3 \times \Z_3$     &  \qquad  18\\[.1cm] 
                    &  \qquad (1 2 3), (1 2), ( 4 5 6), (4 5)  &  \qquad$ S_3 \times S_3$     &  \qquad  36 \\[.2cm] \hline \\[-.1cm]
   2 + 4 &  \qquad (1 2), (3 4 5 6) & \qquad $ \Z_2 \times \Z_4 $ &  \qquad 8 \\[.1cm]                  
                    &  \qquad (1 2), (3 4 5 6), (3 4)     &  \qquad$ \Z_2 \times S_4$     &  \qquad  48 \\[.1cm] 
                    &  \qquad (1 2), (3 4 5), (4 5 6)     &  \qquad$ \Z_2 \times A_4$     &  \qquad  24 \\[.1cm] 
                    &  \qquad (1 2), (3 4 5 6), (3 5)     &  \qquad$ \Z_2 \times (\Z_2 \wr \Z_2)$     &  \qquad  16 \\[.2cm] \hline \\[-.2cm]  
  1 + 5 &  \qquad (1 2 3 4 5) & \qquad $ \Z_5$ &  \qquad 5 \\[.1cm]            
                     &  \qquad (1 2 3 4 5), (1 2 3)     &  \qquad$ A_5$     &  \qquad  60  \\[.1cm]       
                    &  \qquad (1 2 3 4 5), (1 2)     &  \qquad$ S_5$     &  \qquad  120  \\[.2cm] \hline \\[-.2cm]    
   6 &  \qquad (1 2 3 4 5 6) & \qquad $ \Z_6$ &  \qquad 6 \\[.1cm]  
        &    \qquad (1 2 3 4 5), (1 2 3), (1 2)(5 6) & \qquad $ A_6$ &  \qquad 360 \\[.1cm]  
                  &  \qquad (1 2 3 4 5 6), (1 2) & \qquad $ S_6$ &  \qquad 720 \\[.1cm]  
                    &    \qquad (1 2 3)(4 5 6), (1 4) & \qquad $\Z_2 \wr \Z_3$ &  \qquad 24 \\[.1cm]        
                     &    \qquad (1 2 3), (1 4)(2 5)(3 6) & \qquad $\Z_3 \wr \Z_2$ &  \qquad 18 \\[.1cm]     
                     &    \qquad (1 2 3), (1 2), (1 4)(2 5)(3 6) & \qquad $ S_3 \wr \Z_2$ &  \qquad 72 \\[.2cm] \hline     
    \end{tabular}
 
 \end{figure}

The partitions $6 = 1 + 1 + \cdots + 1$ and $6 = 1 + \cdots + 1 + 2$ each give only one group, $(1)$ and $\mathbb{Z}_2 \cong  \langle (12) \rangle$, respectively.  Also, $6 = 1 + 1 + 1 + 3$ yields two: $\mathbb{Z}_3 \cong \langle (1\  2 \ 3) \rangle$ and $S_3 \cong  \langle (1\  2 \ 3), (1 2) \rangle$.    The rest are listed in the table.  This is far from all the groups of degree 6.  In \cite{DM}, there are 16 transitive subgroups of $S_6$ listed.
Our table contains 6.  One familiar transitive subgroup of $S_6$ missing from our list is the dihedral group $D_{12}$.  The dihedral groups
seem to be the next interesting case.

{\bf Problem 1:} For all $n \ge 1$, let $D_{2n}$ denote the dihedral group of symmetries of a regular $n$-gon.  Show that 
$B_n/D_{2n}$ is an SCO.

We know this only for the trivial cases $n = 1, 2, 3$ and, as noted in the table, $n = 4$, since $D_8$ is a
wreath product $\Z_2 \wr \Z_2$.   Note that $D_{2n}$ is the semidirect product $\Z_n \rtimes \Z_2$.  However the dihedral groups
are not wreath products for $n \ge 5$.  This is easy to see: each permutation in $D_{2n}$ has either no fixed points, 
or one (for odd $n$) or two (for even $n$), or $n$ fixed points.    However, a wreath product $K \wr T$ has maps with
$k, 2k, \ldots, tk$ fixed points, where $K$ is a subgroup of $S_k$ and $T$ is a subgroup of $S_t$.

The approach presented above can be extended in some other special cases.  Here is an example.  Let 

\qquad $G$ be the subgroup of $S_8$ generated by $\{ (1 5)(2 6)(3 7)(4 8) , (1 4)(2 3) \}$, and 

\qquad $H$ be the subgroup of $G$ generated by $\{ (1 5)(2 6)(3 7)(4 8) , (1 4)(2 3)(5 8)(6 7) \}$.

Then $G =K \wr T$ where
  
 \qquad $K = \langle (1 2 4 3)^2 \rangle$ is a subgroup of $S_4$ and $T = \langle (12) \rangle$ is a subgroup of $S_2$.   

Thus $B_8 / G$ is an SCO.  One can describe how to refine the $G$-orbits to obtain  $H$-orbits and how to produce the 
additional symmetric chains required for an SCD of $B_8 / H$.   Unfortunately,
 this sort of argument does not appear to yield anything for $D_{2n}$ since it is not a ``manageable" subgroup of any
 group we can handle, such as $Z_n \wr Z_2$ or $Z_2 \wr Z_n$.

As noted in the introduction, Stanley \cite{RPS} first raised the question of whether the distributive lattice $L(k, t)$, the poset with elements all integer sequences $\overline{a} = (a_1, a_2, \ldots, a_t)$ such that $0 \le a_1 \le \cdots \le a_t \le k$  and  componentwise order,
is an SCO for all $k$ and $t$.
He noted \cite{RPS2, RPS3} that $L(k, t)$ is obtained as a quotient of the Boolean lattice $B_{kt}$ by the
group $S_k \wr S_t$ with its natural action on $[k] \times [t]$.  (This is the action described at the beginning of  Section \ref{s:mainresult}.)    It is easiest to see this if we think of $L(k, t)$ as the set of down sets of the grid $\overline{k} \times \overline{t}$ ordered
by containment (using $\overline{m}$ to denote the $m$-element chain).  The wreath product $S_k \wr S_t$ acts on the elements of $\overline{k} \times \overline{t}$ by permuting the
elements of $C_j = \{ (1, j), (2, j), \ldots, (k,j) \}$  independently under $S_k$ for each $j = 1, 2, \ldots, t$, then permuting $C_1, C_2, \ldots, C_t$.  Each
equivalence class under this action contains a unique down set of  $\overline{k} \times \overline{t}$.

If we regard this description in two steps then the independent permutations of the elements in each  $C_j$
produce the quotient $(B_k/S_k)^t$, just as in (\ref{e:inside}) in Section \ref{s:mainresult}.  In this case, $B_k/S_k \cong \overline{k+1}$.  The second part of the action is just $S_t$ acting on the coordinates of $(\overline{k+1})^t$.  We 
restate Stanley's problem.

{\bf Problem 2:}  Show that for all $k, t$, $L(k, t)$ is an SCO.  Equivalently, show that this  quotient is an SCO:
$$ B_{kt} / (S_k \wr S_t ) \ \cong \ (\overline{k+1})^t / S_t .$$
In thinking about Dhand's result, one can ask several types of questions.  If we focus on groups acting by permuting coordinates,
we can pose this problem.

{\bf Problem 3:}  Determine whether $P^n/G$ is an SCO, assuming that $P$ is an SCO and the subgroup $G$ of $S_n$ acts on
$P^n$ by permuting coordinates.

One approach to this follows the strategy given in {\bf \S2.1} and the proof of Theorem \ref{T:dhand}.  With an SCD 
 $C_1, C_2, \ldots, C_s$ of $P$, the family of grids  $C(\overline{j})$, $\overline{j} \in [s]^t$, forms a partition of $P^n$ into
 symmetric, saturated subsets.   Every $\sigma \in G$ permutes the members of the family $\{ C(\overline{j}) \ | \ \overline{j} \in [s]^n \}$
 and the restriction of $\sigma$ defines a local isomorphism of $C(\overline{j})$ to $C(\sigma(\overline{j}))$.  We have an SCD for
 $P^n/G$ if we can obtain an SCD of each $C(\overline{j})/G_{\overline{j}}$ for a set $J$ of representatives of the $G$-orbits of
 $[s]^n$.   If we follow this line, quotients of products of chains could be the key to Problem 3.   Of course, in the case
 that $G = S_n$, the stabilizers $G_{\overline{j}}$ are themselves products of symmetric groups acting on disjoint subsets of
 $[n]$, so we are back to Stanley's problem.
 
 In another direction, we can ask (as we did in the introduction) whether $P/G$ is an SCO, given that $P$ has an SCD and $G$ is a subgroup of $\Au(P)$.  This would seem to be quite different from the problem of considering subgroups of $S_n$ acting by permuting coordinates of
 a product.  A special case is to ask: for which SCOs $P$ is $P/\Au(P)$ an SCO?  Most partially ordered sets have no nontrivial 
 automorphisms, so there is an immediate positive and vacuous answer.  But in general?

\end{document}